\documentclass[twoside,12pt]{article}
 \usepackage{bbm}
 \usepackage{mathrsfs}
  \usepackage{amsfonts}
\usepackage{amsmath}

\newtheorem{theorem}{Theorem}[section]%
\newtheorem{lemma}[theorem]{Lemma}%
\newtheorem{conjecture}[theorem]{Conjecture}

\newcommand{\abs}[1]{\lvert#1\rvert}

\newenvironment{proof}[1][Proof]{\noindent\textit{#1: } }{\hfill\rule{1mm}{2mm}}

\makeatletter \@addtoreset{equation}{section} \makeatother

\begin{document}

  \title{On the minimal feedback arc set of $m$-free Digraphs
 \thanks{Supported by the Key Project of Chinese Ministry of Education (109140) and NNSF of China (No. 11071233).}
  }
 \author{Hao Liang\thanks{Corresponding author: lianghao@mail.ustc.edu.cn}\,\,
 \\
  {\small Department of Mathematics}\\
  {\small Southwestern University of Finance and Economics}\\
  {\small Chengdu {\rm 611130}, China}\\
 \\
  Jun-Ming Xu \\
  {\small School of Mathematical Sciences}\\
  {\small University of Science and Technology of China}\\
 {\small Wentsun Wu Key Laboratory of CAS}\\
  {\small Hefei {\rm 230026}, China}\\
  }
\date{}

\maketitle {\centerline{\bf\sc Abstract}}\vskip 8pt  For a simple
digraph $G$, let $\beta(G)$ be the size of the smallest subset
$X\subseteq E(G)$ such that $G-X$ has no directed cycles, and let
$\gamma(G)$ be the number of unordered pairs of nonadjacent vertices
in $G$. A digraph $G$ is called $m$-free if $G$ has no directed
cycles of length at most $m$. This paper proves that $\beta(G)\leq
\frac{1}{m-2}\gamma(G)$ for any $m$-free digraph $G$, which
generalized some known results.



\section{Introduction}

Let $G=(V,E)$ be a digraph without
loops and parallel edges, where $V=V(G)$ is the vertex-set and
$E=E(G)$ is the edge-set.

It is well known that the cycle rank of an undirected graph $G$ is
the minimum number of edges that must be removed in order to
eliminate all of the cycles in the graph. That is, if $G$ has
$\upsilon$ vertices, $\varepsilon$ edges, and $\omega$ connected
components, then the minimum number of edges whose deletion from $G$
leaves an acyclic graph equals the cycle rank (or Betti number)
$\rho(G)=\varepsilon-\upsilon+\omega$ (see Xu~\cite{x03}). However,
the same problem for a digraph is quite difficulty. In fact, the
Betti number for a digraph was proved to be NP-complete by Karp in
1972 (see the 8th of 21 problems in~\cite{k72}).

A digraph $G$ is called to be {\it $m$-free} if there is no directed
cycle of $G$ with length at most $m$. We say $G$ is {\it acyclic} if
it has no directed cycles. For a digraph $G$, let $\beta(G)$ be the
size of the smallest subset $X\subseteq E(G)$ such that $G-X$ is
acyclic, here $X$ is called a{\it  minimal feedback arc-set} of $G$.
Let $\gamma(G)$ be the number of unordered pairs of nonadjacent
vertices in $G$, called the {\it number of missing edges} of $G$.

Chudnovsky, Seymour, and Sullivan~\cite{css08} proved that
$\beta(G)\leq \gamma(G)$ if $G$ is a $3$-free digraph and gave the
following conjecture.
\begin{conjecture}\label{cjt1.1}
If $G$ is a $3$-free digraph, then $\beta(G)\leq
\frac{1}{2}\gamma(G)$.
\end{conjecture}

Concerning this conjecture, Dunkum, Hamburger, and
P\'{o}r~\cite{dhp11} proved that $\beta(G)\leq 0.88\gamma(G)$. Very
recently, Chen et al.~\cite{ckls11} improved the result to
$\beta(G)\leq 0.8616\gamma(G)$. Conjecture~\ref{cjt1.1} is closely
related to the following special case of the conjecture proposed by
Caccetta and H\"{a}ggkvist~\cite{ch78}.

\begin{conjecture}\label{cjt1.2}
Any digraph on n vertices with minimum out-degree at least $n/3$
contains a directed triangle.
\end{conjecture}

Short of proving the conjecture, one may seek as small a value of
$c$ as possible such that every digraph on $n$ vertices with minimum
out-degree at least $cn$ contains a triangle. This was the strategy
of Caccetta and H\"aggkvist~\cite{ch78}, who obtained the value
$c\le 0.3819$. Bondy~\cite{b97} showed that $c\le 0.3797$, and
Shen~\cite{s98} improved it to $c\le 0.3542$. By using a result of
Chudnovsky, Seymour and Sullivan~\cite{css08} related to
Conjecture~\ref{cjt1.1}, Hamburger, Haxell, and
Kostochka~\cite{hhk07} further improved this bound to $0.35312$.
Namely, any digraph on $n$ vertices with minimum out-degree at least
$0.35312n$ contains a directed triangle.

More generally, Sullivan~\cite{s08} proposed the following
conjecture, and gave an example showing that this would be best
possible if this conjecture is true. Conjecture~\ref{cjt1.1} is the
special case when $m=3$.
\begin{conjecture}\label{cjt1.3}
If $G$ is an $m$-free digraph with $m\geq 3$, then
$$\beta(G)\leq \frac{2}{(m+1)(m-2)}\gamma(G).$$
\end{conjecture}

Sullivan proved partial results of Conjecture~\ref{cjt1.3}, and
showed that $\beta(G)\leq \frac{1}{m-2}\gamma(G)$ for $m=4,5$. In
this article, we prove the following theorem, which extends
Sullivan's result to more general $m$-free digraphs for $m\geq 4$.

\begin{theorem}\label{thm1.4}
If $G$ is an m-free digraph with $m\geq 4$, then $\beta(G)\leq
\frac{1}{m-2}\gamma(G)$.
\end{theorem}

\section{Some Lemmas}

Let $G$ be a simple digraph. For two disjoint subsets $A,B\subseteq
V (G)$, let $E(A,B)$ denote the set of directed edges from $A$ to
$B$, that is, $E(A,B)=\{(a,b)|\ a\in A,\ b\in B\}$. Let
$\bar{E}(A,B)$ be the missing edges between $A$ and $B$. It follows
that\emph{}
 $$
 \abs {\bar{E}(A,B)}=\abs {\bar{E}(B,A)}=|A|\cdot |B|-|E(A,B)|-|E(B,A)|.
 $$

A directed $(v_0,v_k)$-path $P$ in $G$ is a sequence of distinct
vertices $(v_0,v_1,\cdots, v_{k-1},v_k)$, where $(v_i,v_{i+1})$ is a
directed edge for each $i=0, \cdots, k-1$, its length is $k$.
Clearly, the subsequence $(v_1,\cdots, v_{k-1})$ is a
$(v_1,v_{k-1})$-path, denoted by $P'$. We can denote
$P=(v_0,P',v_k)$. A directed path $P$ is said to be {\it induced} if
every edge in the subgraph induced by vertices of $P$ is contained
in $P$.

For $v\in V(G)$, let $N_{i}^+(v)$ be the set of vertices $u$ such
that the shortest directed $(v,u)$-path has length $i$. Similarly,
let $N_{i}^-(v)$ be the set of vertices whose shortest directed path
to $v$ has length $i$. An induced directed $(v_0,v_k)$-path is
called to be {\it shortest} if $v_k\in N_{k}^+(v_0)$. From
definition, we immediately have the following result.

\begin{lemma}\label{lem2.1}
If $(v_0,v_1,\cdots, v_{k-1},v_k)$ is a shortest induced directed
$(v_0,v_k)$-path, then for any $i$ and $j$ with $0\leq i< j\leq k$,
$$v_j\in N_{j-i}^+(v_i) \ \ \text{and} \ \ v_i\in N_{j-i}^-(v_j).$$
\end{lemma}

Let $\mathscr{P}(G)$ be the set of shortest induced directed paths
of $G$, and $m$ be a positive integer with $m\geq 4$. Let $v\in
V(G)$ and $k$ be an integer with $1\leq k\leq m-3$. For any $P\in
\mathscr{P}(G)$ of length $k-1$ and $x,y,z\in V(G)$, set

$P_k(v)=\{(x,y,z)|\ (x,P,y,z)\in \mathscr{P}(G), x=v\}$ and
$p_k(v)=|P_k(v)|$,

$Q_k(v)=\{(x,y,z)|\ (x,P,y,z)\in \mathscr{P}(G), y=v\}$ and
$q_k(v)=|Q_k(v)|$,

$R_k(v)=\{(x,y,z)|\ (x,P,y,z)\in \mathscr{P}(G), z=v\}$ and
$r_k(v)=|R_k(v)|$.

$P'_k(v)=\{(x,y,z)|\ (x,y,P,z)\in \mathscr{P}(G), x=v\}$ and
$p'_k(v)=|P'_k(v)|$,

$Q'_k(v)=\{(x,y,z)|\ (x,y,P,z)\in \mathscr{P}(G), y=v\}$ and
$q'_k(v)=|Q'_k(v)|$,

$R'_k(v)=\{(x,y,z)|\ (x,y,P,z)\in \mathscr{P}(G), z=v\}$ and
$r'_k(v)=|R'_k(v)|$.

\begin{lemma}\label{lem2.2}
For any integer $k$ with $1\leq k\leq m-3$ and $P\in\mathscr{P}(G)$
of length $k-1$,
  \begin{equation}\label{e2.1}
  \sum\limits_{v\in V(G)} p_k(v)=\sum\limits_{v\in V(G)} q_k(v)=\sum\limits_{v\in V(G)} r_k(v),
  \end{equation}
and
  \begin{equation}\label{e2.2}
  \sum\limits_{v\in V(G)} p'_k(v)=\sum\limits_{v\in V(G)} q'_k(v)=\sum\limits_{v\in V(G)} r'_k(v).
 \end{equation}
\end{lemma}

\begin{proof}
For each integer $k$ with $1\leq k\leq m-3$ and $P\in
\mathscr{P}(G)$ of length $k-1$,
 $$
 \sum_{v\in V(G)}p_k(v),\  \sum_{v\in V(G)} q_k(v),\ \sum_{v\in V(G)} r_k(v)
 $$
are all equal to the number of triples $(x,y,z)$ of distinct
vertices such that $(x,P,y,z)\in \mathscr{P}(G)$ for $P\in
\mathscr{P}(G)$. Thus (\ref{e2.1}) holds. The proof of (\ref{e2.2})
is similar.
\end{proof}

\begin{lemma}\label{lem2.3}
If $G$ is an $m$-free digraph, then for any $v\in V(G)$ and any
integer $k$ with $1\leq k\leq m-3$,
$$\left\{\begin{array}{l}
p_k(v)=\abs{E(N_{k+1}^+(v),N_{k+2}^+(v))},\\
q_k(v)\leq \abs{\bar{E}(N_{k+1}^-(v),N_{1}^+(v))},\\
r_k(v)\leq \abs{\bar{E}(N_{1}^-(v),N_{k+2}^-(v))},\\
p'_k(v)\leq \abs{\bar{E}(N_{1}^+(v),N_{k+2}^+(v))},\\
q'_k(v)\leq \abs{\bar{E}(N_{k+1}^+(v),N_{1}^-(v))},\\
r'_k(v)=\abs{E(N_{k+2}^-(v),N_{k+1}^-(v))}.
\end{array}\right.$$
\end{lemma}

\begin{proof} By definition, for each edge
$(u,w)\in E(N_{k+1}^+(v),N_{k+2}^+(v))$, there exists $v_i\in
N_{i}^+(v)$, for each $i=1,2,\cdots k$, such that
$(v,v_1,\cdots,v_{k-1},v_k,u,w)$ is a directed $(v,w)$-path of
length $k+2$. Since $G$ is $m$-free and $1\leq k\leq m-3$, it is
easy to see that $(v,v_1,\cdots, v_{k-1},v_k,u,w)$ is a shortest
induced directed path. It follows that $(v,u,w)\in P_k(v)$ and
 \begin{equation}\label{e2.3}
 p_k(v)\geq \abs{E(N_{k+1}^+(v),N_{k+2}^+(v))}.
 \end{equation}

On the other hand, for each $(v,u,w)\in P_k(v)$, from the definition
of $P_k(v)$ and Lemma~\ref{lem2.1}, $u\in N_{k+1}^+(v)$ and $w\in
N_{k+2}^+(v)$. Thus $(u,w)\in E(N_{k+1}^+(v),N_{k+2}^+(v))$. It
follows that
 \begin{equation}\label{e2.4}
 p_k(v)\leq \abs{E(N_{k+1}^+(v),N_{k+2}^+(v))}.
 \end{equation}
Combining (\ref{e2.3}) and (\ref{e2.4}), we have that
$p_k(v)=\abs{E(N_{k+1}^+(v),N_{k+2}^+(v))}$. The proof of
$r'_k(v)=\abs{E(N_{k+2}^-(v),N_{k+1}^-(v))}$ is similar.

For each $(u,v,w)\in Q_k(v)$, from the definition of $Q_k(v)$ and
Lemma~\ref{lem2.1}, we have $u\in N_{k+1}^-(v)$, $w\in N_{1}^+(v)$
and $uw\notin E(G)$. Since $G$ is $m$-free, we have $(w,u)\notin
E(G)$. If not, there exists a directed cycle $(v,w,u)\cdots v$ with
length $l=k+3\leq m$, a contradiction. So $(u,w)\in
\abs{\bar{E}(N_{k+1}^-(v),N_{1}^+(v))}$. Thus, $q_k(v)\leq
\abs{\bar{E}(N_{k+1}^-(v),N_{1}^+(v))}$. The proof of $q'_k(v)\leq
\abs{\bar{E}(N_{k+1}^+(v),N_{1}^-(v))}$ is similar.

For each $(u,w,v)\in R_k(v)$, from the definition of $R_k(v)$ and
Lemma~\ref{lem2.1}, we have $u\in N_{k+2}^-(v)$, $w\in N_{1}^-(v)$
and $(u,w)\notin E(G)$. Since $G$ is $m$-free, $(w,u)\notin E(G)$.
Otherwise, there exists a directed cycle $(w,u,\cdots, w)$ with
length $l=k+2\leq m-1$, a contradiction. Thus we have $(u,w)\in
\abs{\bar{E}(N_{1}^-(v),N_{k+2}^-(v))}$. It derives that $r_k(v)\leq
\abs{\bar{E}(N_{1}^-(v),N_{k+2}^-(v))}$. The proof of $p'_k(v)\leq
\abs{\bar{E}(N_{1}^+(v),N_{k+2}^+(v))}$ is similar.
\end{proof}

\vskip6pt

For any $v\in V(G)$ and any integer $k$ with $1\leq k\leq m-3$, set
 $$
 \alpha_k(v)=\frac{p_k(v)}{s_k(v)} \ \ \text{and} \ \ \beta_k(v)=\frac{r'_k(v)}{t_k(v)}.
 $$

Here
 \begin{equation}\label{e2.5}
 s_k(v)=\sum\limits_{i=k}^{m-3} p'_i(v)+\sum\limits_{i=1}^{k} q'_i(v)\
 \ \text{and} \ \
 t_k(v)=\sum\limits_{i=k}^{m-3} r_i(v)+\sum\limits_{i=1}^{k} q_i(v).
 \end{equation}

The result is obvious.
\begin{lemma}\label{lem2.4}
If $a_i\geq 0, b_i\geq 0$ for each $i=1, 2, \cdots, n$, and
$\sum\limits_{i=1}^n b_i>0$, then
 $$
 \min\limits_{1\leq i\leq n}\left\{\frac{a_i}{b_i}\right\}\leq
 \frac{\sum\limits_{i=1}^n a_i}{\sum\limits_{i=1}^n b_i}.
 $$
\end{lemma}

Let
  \begin{equation}\label{e2.6}
 \alpha=\min\limits_{v\in V(G) \atop 1\leq k\leq m-3}\{ \alpha_k(v)\}\ \ {\rm and}\ \
 \beta=\min\limits_{v\in V(G) \atop 1\leq k\leq m-3}\{ \beta_k(v)\}.
 \end{equation}
Applying Lemma~\ref{lem2.4}, we obtain the following bound about
$\alpha$ and $\beta$.

\begin{lemma}\label{lem2.5}
If $G$ is a $m$-free digraph, then
 $$
 \min\{ \alpha, \beta\}\leq \frac{1}{m-2}.
 $$

\end{lemma}

\begin{proof}
By Lemma~\ref{lem2.4}, we have
 $$
 \alpha=\min\limits_{v\in V(G) \atop 1\leq k\leq m-3}\{ \alpha_k(v)\}
 =\min\limits_{v\in V(G) \atop 1\leq k\leq m-3}\left\{ \frac{p_k(v)}{s_k(v)}\right\}
 \leq \frac{\sum\limits_{k=1}^{m-3}\sum\limits_{v\in V(G)} p_k(v)}{\sum\limits_{k=1}^{m-3}\sum\limits_{v\in V(G)} s_k(v)},
 $$
and
 $$
 \beta=\min\limits_{v\in V(G) \atop 1\leq k\leq m-3}\{ \beta_k(v)\}
 =\min\limits_{v\in V(G) \atop 1\leq k\leq m-3}\left\{ \frac{r'_k(v)}{t_k(v)}\right\}
 \leq \frac{\sum\limits_{k=1}^{m-3}\sum\limits_{v\in V(G)} r'_k(v)}{\sum\limits_{k=1}^{m-3}\sum\limits_{v\in V(G)} t_k(v)}.
 $$
It follows that
 \begin{equation}\label{e2.7}
 \min\{ \alpha, \beta\}\leq  \frac{\sum\limits_{k=1}^{m-3}\left(\sum\limits_{v\in V(G)} p_k(v)
 +\sum\limits_{v\in V(G)} r'_k(v)\right)}{\sum\limits_{k=1}^{m-3}\left(\sum\limits_{v\in V(G)} s_k(v)
 +\sum\limits_{v\in V(G)} t_k(v)\right)}.
 \end{equation}

Summing $s_k(v)$ and $t_k(v)$ over all $v\in V(G)$ and noting
(\ref{e2.5}), we have
$$
 \begin{array}{ll}
   \sum\limits_{k=1}^{m-3}\sum\limits_{v\in V(G)} s_k(v)
   &=\sum\limits_{k=1}^{m-3}\left(\sum\limits_{i=k}^{m-3} \sum\limits_{v\in V(G)} p'_i(v)\right)
   +\sum\limits_{k=1}^{m-3}\left(\sum\limits_{i=1}^{k} \sum\limits_{v\in V(G)}q'_i(v)\right)\\
   &=\sum\limits_{k=1}^{m-3}\left(\sum\limits_{i=k}^{m-3} \sum\limits_{v\in V(G)} r'_i(v)\right)
   +\sum\limits_{k=1}^{m-3}\left(\sum\limits_{i=1}^{k} \sum\limits_{v\in V(G)}r'_i(v)\right)\\
   &=\sum\limits_{k=1}^{m-3} \left(\sum\limits_{i=1}^{m-3} \sum\limits_{v\in V(G)} r'_i(v)
   +\sum\limits_{v\in V(G)} r'_k(v)\right)\\
   &=(m-2)\sum\limits_{k=1}^{m-3} \sum\limits_{v\in V(G)} r'_k(v)
\end{array}
 $$

and
$$
 \begin{array}{ll}
   \sum\limits_{k=1}^{m-3}\sum\limits_{v\in V(G)} t_k(v)
   &=\sum\limits_{k=1}^{m-3}\left(\sum\limits_{i=k}^{m-3} \sum\limits_{v\in V(G)} r_i(v)\right)
   +\sum\limits_{k=1}^{m-3}\left(\sum\limits_{i=1}^{k} \sum\limits_{v\in V(G)}q_i(v)\right)\\
   &=\sum\limits_{k=1}^{m-3} \left(\sum\limits_{i=k}^{m-3} \sum\limits_{v\in V(G)} p_i(v)\right)
   +\sum\limits_{k=1}^{m-3}\left(\sum\limits_{i=1}^{k} \sum\limits_{v\in V(G)}p_i(v)\right)\\
   &=\sum\limits_{k=1}^{m-3}\left(\sum\limits_{i=1}^{m-3} \sum\limits_{v\in V(G)} p_i(v)
   +\sum\limits_{v\in V(G)} p_k(v)\right)\\
   &=(m-2)\sum\limits_{k=1}^{m-3} \sum\limits_{v\in V(G)} p_k(v)
\end{array}
 $$
It follows that
 $$
 \sum\limits_{k=1}^{m-3}\left(\sum\limits_{v\in V(G)} s_k(v)+\sum\limits_{v\in V(G)} t_k(v)\right)
 =(m-2)\sum\limits_{k=1}^{m-3}\left(\sum\limits_{v\in V(G)} p_k(v)+\sum\limits_{v\in V(G)} r'_k(v)\right).
 $$
Substituting this equality into (\ref{e2.7}) yields
 $$
 \min\{ \alpha, \beta\}\leq \frac{1}{m-2}.
 $$

The lemma follows.
\end{proof}

\section{Proof of Theorem~\ref{thm1.4}}

Clearly Theorem~\ref{thm1.4} holds for $\abs {V(G)}\leq m$. We
proceeds the proof by induction on $\abs {V(G)}$ under the
assumption that Theorem~\ref{thm1.4} holds for all digraphs with
$\abs {V(G)}< n$, here $n> m$. Now let $G$ be an $m$-free digraph
with $\abs {V(G)}=n$, we may assume that for any $v\in V(G)$,
$N_{1}^+(v)\neq \emptyset$ and $N_{1}^-(v)\neq \emptyset$.
Otherwise, if there exists $v\in V(G)$ such that
$N_{1}^+(v)=\emptyset$ or $N_{1}^-(v)=\emptyset$, then $v$ is not in
a directed cycle. From the inductive hypothesis, we can choose
$X\subseteq E(G-v)$ with $\abs {X}\leq \frac{1}{m-2}\gamma(G-v)$
such that $(G-v)-X$ is acyclic, then $G-X$ has no directed cycles.
It follows that
 $$
 \beta(G)\leq \abs {X}\leq \frac{1}{m-2}\gamma(G-v)\leq \frac{1}{m-2}\gamma(G).
 $$

From Lemma~\ref{lem2.5}, we have that $\alpha\leq \frac{1}{m-2}$ or
$\beta\leq \frac{1}{m-2}$. For each case, we prove that there exists
$X\subseteq E(G)$ satisfying $\abs X\leq \frac{1}{m-2}\gamma (G)$
and $G -X$ has no directed cycles.

\noindent{\bf Case 1.}  $\alpha\leq \frac{1}{m-2}$.

By (\ref{e2.6}), there exist a vertex $v\in V(G)$ and an integer $k$
with $1\leq k\leq m-3$ such that
 $$
 \alpha=\alpha_k(v)=\frac{p_k(v)}{s_k(v)}\leq \frac{1}{m-2}.
 $$
We consider the partition $\{V_1,V_2\}$ of $V(G)$, where
 $$
 V_1=\bigcup\limits_{i=1}^{k+1} N_{i}^+(v),\ \  V_2=V(G)\setminus V_1.
 $$
Clearly, $N_{1}^-(v)\subset V_2$ and $\bigcup_{i=k+2}^{m-1}
N_{i}^+(v)\subset V_2$. Since $G$ is an $m$-free digraph, we claim
 $$
 N_{1}^-(v)\cap \bigcup_{i=1}^{m-1} N_{i}^+(v)=\emptyset.
 $$
Otherwise, let $u\in N_{1}^-(v)\cap \bigcup_{i=1}^{m-1} N_{i}^+(v)$.
Then $(u,v)\in E(G)$ and there exists a directed $(v,u)$-path $P$
with length $l_1\leq m-1$. Then $P+(u,v)$ is a directed cycle with
length $l_1+1\leq m$, a contradiction.

Thus the number of missing edges between $V_1$ and $V_2$ satisfies
$$
\begin{array}{ll}
 \abs {\bar{E}(V_1,V_2)}
 &\geq \abs {\bar{E}(\ \bigcup_{i=1}^{k+1} N_{i}^+(v),N_{1}^-(v)\cup \left(\bigcup_{i=k+2}^{m-1} N_{i}^+(v)\right))}\\
 &\geq \sum\limits_{i=k+2}^{m-1} \abs{\bar{E}(N_{1}^+(v),N_{i}^+(v))}+\sum\limits_{i=2}^{k+1} \abs{\bar{E}(N_{i}^+(v),N_{1}^-(v))}\\
 &\geq \sum\limits_{i=k}^{m-3} p'_i(v)+\sum\limits_{i=1}^{k} q'_i(v)\\
 &=s_k(v).
 \end{array}
 $$
It follows that
 \begin{equation}\label{e3.1}
 \gamma (G)=\gamma(G_1)+\gamma(G_2)+\abs {\bar{E}(V_1,V_2)}\geq \gamma(G_1)+\gamma(G_2)+s_k(v).
 \end{equation}

Let $G_i$ be the induced subgraph by $V_i$ for each $i=1,2$. Since
$\abs {V_1}< n$ and $\abs {V_2}< n$, from the inductive hypothesis,
we have $\beta(G_1)\leq \frac{1}{m-2}\gamma(G_1)$ and
$\beta(G_2)\leq \frac{1}{m-2}\gamma(G_2)$. We choose $X_i\subseteq
E(G_i)$ with
 \begin{equation}\label{e3.2}
 \abs {X_i}\leq \frac{1}{m-2}\gamma(G_i)\ \ {\rm for\ each}\ i=1,2
 \end{equation}
such that $G_i-X_i$ is acyclic.

Let $X_3=E(V_1,V_2)$. Then
$X_3=E(N_{k+1}^+(v),V_2)=E(N_{k+1}^+(v),N_{k+2}^+(v))$, and
 \begin{equation}\label{e3.3}
 \abs {X_3}=|E(N_{k+1}^+(v),N_{k+2}^+(v))|=p_k(v).
  \end{equation}
Let $X=X_1\cup X_2\cup X_3$. Then $G-X$ has no directed cycles and,
by (\ref{e3.1}) $\sim$ (\ref{e3.3}),
 $$
 \begin{array}{ll}
 \abs {X}&= \abs {X_1}+\abs {X_2}+\abs {X_3}\\
        &=\abs {X_1}+\abs {X_2}+p_k(v)\\
        &\leq \frac{1}{m-2}\gamma(G_1)+\frac{1}{m-2}\gamma(G_2)+\frac{1}{m-2}s_k(v)\\
        &= \frac{1}{m-2}\left(\gamma(G_1)+\gamma(G_2)+s_k(v)\right)\\
        &\leq \frac{1}{m-2}\gamma (G).\\
\end{array}$$

\noindent{\bf Case 2.} $\beta\leq \frac{1}{m-2}$.

By (\ref{e2.6}), there exist a vertex $v\in V(G)$ and an integer $k$
with $1\leq k\leq m-3$ such that
 $$
 \beta=\beta_k(v)=\frac{r'_k(v)}{t_k(v)}\leq \frac{1}{m-2}.
 $$
We consider the partition $\{V_1,V_2\}$ of $V(G)$, where
 $$
 V_1=\bigcup\limits_{i=1}^{k+1} N_{i}^-(v),\ \  V_2=V(G)\setminus V_1.
 $$
Clearly, $N_{1}^+(v)\subset V_2$, $\bigcup_{i=k+2}^{m-1}
N_{i}^-(v)\subset V_2$ and $N_{1}^+(v)\cap \bigcup_{i=k+2}^{m-1}
N_{i}^-(v)=\emptyset$. The number of missing edges between $V_1$ and
$V_2$ satisfies
$$
\begin{array}{ll}
 \abs {\bar{E}(V_1,V_2)}
 &\geq \abs {\bar{E}(\ \bigcup_{i=1}^{k+1} N_{i}^-(v),N_{1}^+(v)\cup \left(\bigcup_{i=k+2}^{m-1} N_{i}^-(v)\right))}\\
 &\geq \sum\limits_{i=k+2}^{m-1} \abs{\bar{E}(N_{1}^-(v),N_{i}^-(v))}+\sum\limits_{i=2}^{k+1} \abs{\bar{E}(N_{i}^-(v),N_{1}^+(v))}\\
 &\geq \sum\limits_{i=k}^{m-3} r_i(v)+\sum\limits_{i=1}^{k} q_i(v)\\
 &=t_k(v).
 \end{array}
 $$

Then
 $$
 \gamma (G)=\gamma(G_1)+\gamma(G_2)+\abs {\bar{E}(V_1,V_2)}\geq \gamma(G_1)+\gamma(G_2)+t_k(v).
 $$

Let $G_i$ be the induced subgraph by $V_i$ for each $i=1,2$. For
$i=1,2$, from the inductive hypothesis, $\beta(G_1)\leq
\frac{1}{m-2}\gamma(G_1)$ and $\beta(G_2)\leq
\frac{1}{m-2}\gamma(G_2)$, we can choose $X_i\subseteq E(G_i)$ with
$\abs {X_i}\leq \frac{1}{m-2}\gamma(G_i)$ such that $G_i-X_i$ is
acyclic. Let $X_3=(V_2,V_1)$, we have
$X_3=E(V_2,N_{k+1}^-(v))=E(N_{k+2}^-(v),N_{k+1}^-(v))$, and $\abs
{X_3}=r'_k(v)$. Let $X=X_1\cup X_2\cup X_3$. Then $G-X$ has no
directed cycles. Hence
 $$
 \begin{array}{ll}
 \abs {X}&= \abs {X_1}+\abs {X_2}+\abs {X_3}\\
        &=\abs {X_1}+\abs {X_2}+r'_k(v)\\
        &\leq \frac{1}{m-2}\gamma(G_1)+\frac{1}{m-2}\gamma(G_2)+\frac{1}{m-2}t_k(v)\\
         &= \frac{1}{m-2}\left(\gamma(G_1)+\gamma(G_2)+t_k(v)\right)\\
        &\leq \frac{1}{m-2}\gamma (G).\\
\end{array}$$

For each case, there exists $X\subseteq E(G)$ satisfying $\abs X\leq
\frac{1}{m-2}\gamma (G)$ and $G-X$ has no directed cycles. This
implies that $\beta(G)\leq \abs X\leq \frac{1}{m-2}\gamma (G)$, and
Theorem~\ref{thm1.4} follows.

\end{document}